\documentclass[11pt,letterpaper,reqno]{amsart}
\usepackage{amsmath,amssymb,amsthm}
\usepackage{hyperref}

\newtheorem{theorem}{Theorem}[section]

\newtheorem{lemma}[theorem]{Lemma}
\theoremstyle{remark}

\numberwithin{equation}{section}
\newcommand{\R}{\mathbb R}
\newcommand{\loc}{\mathrm{loc}}

\title[Half-space theorem for the Allen--Cahn equation]
{A half-space theorem for the Allen--Cahn equation in all dimensions}

\begin{document}

\author{Sanghoon Lee}
\address[Sanghoon Lee]{Research Institute of Mathematics, Seoul National University, 1 Gwanak-ro, Gwanak-gu, Seoul 08826, Republic of Korea}
\email{lsh514@snu.ac.kr}

\begin{abstract}
We prove that, in every dimension, a bounded, non-constant entire solution
of the Allen--Cahn equation whose zero set is contained in a half-space
is a one-dimensional heteroclinic solution.
\end{abstract}

\date{\today}

\maketitle

\section{Introduction}\label{sec:intro}

The classical half-space theorem of Hoffman and Meeks \cite{HM}
states that a connected, properly immersed minimal surface without
boundary in $\R^3$ that is contained in a half-space must be a plane. The situation changes in higher dimensions. For every $n\ge4$, the
rotational catenoid in $\R^n$ is a complete, properly embedded,
non-flat minimal hypersurface contained in a slab between two parallel
hyperplanes. Consequently, confinement to a half-space, or even to a slab,
does not by itself imply flatness for minimal hypersurfaces in higher
dimensions. This distinction provides a natural setting in which to
examine the similarities and differences between minimal hypersurfaces
and interfaces arising in phase transition models.

In this paper, we study bounded entire solutions of the Allen--Cahn
equation
\begin{equation}\label{eq:AC}
-\Delta u=u-u^3\qquad\text{in} \quad \R^n.
\end{equation}
The basic non-constant solution is the one-dimensional heteroclinic solution
\[
H(t)=\tanh\!\left(\frac{t}{\sqrt2}\right),
\]
which connects the two constant states $-1$ and $1$. The half-space problem for \eqref{eq:AC} was studied by Hamel, Liu,
Sicbaldi, Wang, and Wei \cite{HLSWW}. In higher dimensions, Agudelo, del Pino, and Wei \cite{ADW} constructed catenoidal solutions whose nodal sets approximate higher-dimensional catenoids on compact sets. These nodal sets have logarithmically growing ends and are therefore not contained in any slab. Moreover, the higher-dimensional counterexamples to De Giorgi's conjecture constructed in \cite{DPKW} have nodal sets that are not contained in any half-space. These observations motivated the question of whether the half-space theorem extends to higher dimensions. In this paper, we answer this question affirmatively:

\begin{theorem}\label{thm:all}
Let $n\ge2$, and let $u\in C^2(\R^n)$ be a bounded, non-constant
solution of \eqref{eq:AC}. If
\[
\{u=0\}\subset\{x_n>0\},
\]
then there exists $a>0$ such that
\[
u(x)=\pm H(x_n-a).
\]
\end{theorem}

Our proof is closely related to the classical work of Berestycki, Caffarelli, and Nirenberg \cite{BCN}, who used the sliding method to establish monotonicity of positive solutions to semilinear elliptic equations in Lipschitz epigraphs, including equations of Allen--Cahn type. Inspired by their work, we establish monotonicity in the $x_n$ direction (see Theorem~\ref{thm:monotonicity}) using the moving plane method. More recently, Beuvin, Farina, and Sciunzi \cite{BFS} established monotonicity results in broad classes of continuous epigraphs bounded from below, using comparison principles in unbounded domains and a modified moving plane argument.

Our result is also related to Gibbons' conjecture, which asserts that a bounded entire solution of the Allen--Cahn equation satisfying
\[
u(x',x_n)\longrightarrow\pm1
\qquad\text{as }x_n\to\pm\infty,
\quad\text{uniformly in }x',
\]
must be one-dimensional in every dimension, without any monotonicity assumption. Berestycki, Hamel, and Monneau \cite{BHM} proved this assertion by the sliding method for a broader class of semilinear equations. A closely related geometric rigidity result was established by Farina \cite{Farina}, who proved that a non-constant entire solution of the Allen--Cahn equation is one-dimensional whenever its zero level set is contained in a slab. For related symmetry and rigidity results, see also \cite{BCN2, FV,BF}.

The comparison theorem in \cite{HLSWW} yields uniform convergence to $-1$ as $x_n\to-\infty$ and provides vertical control on limits of horizontal translates of the solution. These facts allow us to initiate the moving plane argument. Once monotonicity is established, we show that the solution has finite energy density at infinity. The varifold convergence theory of Hutchinson and Tonegawa \cite{HT}, combined with the classical rigidity theorem for stationary cones contained in a half-space, implies that the blow-down limit is a hyperplane, possibly with higher multiplicity. We then use monotonicity to rule out higher multiplicity and apply Wang's rigidity theorem \cite[Theorem~11.1]{Wang}.

Finally, rigidity for \eqref{eq:AC} is closely related to De Giorgi's
conjecture \cite{DeGiorgi}, which predicts that bounded solutions
strictly monotone in one direction are one-dimensional for $n\le8$.
The conjecture was proved in dimensions two and three by
Ghoussoub--Gui \cite{GG} and Ambrosio--Cabr\'e \cite{AC}, respectively,
and by Savin \cite{Savin} for $n\le8$ under the additional assumption
of pointwise limits $\pm1$ in the monotone direction.
Counterexamples in dimensions $n\ge9$ were constructed by
del Pino--Kowalczyk--Wei \cite{DPKW}. Recent developments include
one-dimensionality of stable solutions in $\R^4$ with cubic energy
growth \cite{FS}, extended to solutions of finite Morse index under
the same energy bound in \cite{Florit26}. The classification of bounded
stable solutions in $\R^3$ without energy-growth assumptions was
established independently in \cite{CFFFS,LLWWWW}, yielding
De Giorgi's conjecture in dimension four.

We also mention recent rigidity results for minimal hypersurfaces under
height-growth, half-space, or slab confinement assumptions, together
with suitable stability or topological conditions \cite{CM26,AM26}.
These developments leave many interesting questions open.

\textbf{AI disclosure.} The author used AI tools to help organize the author's original proofs
and improve the English exposition. The author has reviewed and verified
all mathematical arguments and takes full responsibility for the content
of the manuscript.

\section{Preliminaries}\label{sec:prelim}

Throughout this paper, we write $x=(x',x_n)\in\R^n$, where $x'\in\R^{n-1}$.

In this section, we recall two key theorems needed for our analysis.
We begin with the following comparison theorem of Hamel--Liu--Sicbaldi--Wang--Wei \cite{HLSWW}.

\begin{theorem}[Hamel--Liu--Sicbaldi--Wang--Wei
{\cite[Theorem~3]{HLSWW}}]\label{thm:comparison}
Let $n\ge2$, and let $u$ be a bounded, non-constant entire solution of
\eqref{eq:AC}. If $u<0$ in $\{x_n<0\}$, then there exists $a\in\R$ such that
\begin{equation}\label{eq:comparison}
u(x',x_n)\le H(x_n+a).
\end{equation}
Either equality holds identically, or a sequence of horizontal translates
of $u$ converges to $H(x_n+a)$ in $C^2_{\loc}(\R^n)$.
\end{theorem}

In particular, this theorem implies that $\sup_{\R^n}u=1$.
We will also apply \eqref{eq:comparison} to non-constant limits of
horizontal translates to carry out the moving plane method.

Next, we recall the following varifold convergence theorem of Hutchinson--Tonegawa.

Let
\[
F(s):=\frac14(1-s^2)^2,
\qquad
\sigma:=\int_{-1}^{1}\sqrt{2F(s)}\,ds.
\]
For an open set $\Omega\subset\mathbb{R}^n$, we define the energy by
\[
\mathcal{E}(u,\Omega)
:=
\int_\Omega
\left(\frac12|\nabla u|^2+F(u)\right)\,dx.
\]
\begin{theorem}[Hutchinson--Tonegawa
{\cite[Theorem~1]{HT}}]
\label{thm:HT}
Let $\Omega\subset\R^n$ be a bounded Lipschitz domain, and let
$\varepsilon_j\downarrow0$ and $v_j\in C^3(\Omega)$ satisfy
\[
\varepsilon_j\Delta v_j=\varepsilon_j^{-1}F'(v_j),
\qquad |v_j|\le1,
\qquad
\sup_j\int_\Omega
\left(
\frac{\varepsilon_j}{2}|\nabla v_j|^2
+\frac{F(v_j)}{\varepsilon_j}
\right)\,dx<\infty.
\]
Define the associated varifolds by
\[
V_j(\zeta)=\frac1\sigma
\int_{\Omega\cap\{|\nabla v_j|>0\}}
\zeta(x,I-\nu_j\otimes\nu_j)
\varepsilon_j|\nabla v_j|^2\,dx,
\qquad
\nu_j=\frac{\nabla v_j}{|\nabla v_j|},
\]
for $\zeta\in C_c(\Omega\times G(n,n-1))$, where each plane is
identified with its orthogonal projection.
After passing to a subsequence, $V_j$ converges locally as varifolds
to a stationary integral $(n-1)$-varifold $V$, and
\begin{equation}\label{eq:energy-measure}
\left(
\frac{\varepsilon_j}{2}|\nabla v_j|^2
+\frac{F(v_j)}{\varepsilon_j}
\right)\,dx
\rightharpoonup\sigma\|V\|,
\end{equation}
\begin{equation}\label{eq:equipartition}
\frac{F(v_j)}{\varepsilon_j}
-\frac{\varepsilon_j}{2}|\nabla v_j|^2
\longrightarrow0
\quad\text{in }L^1_{\loc}(\Omega).
\end{equation}
\end{theorem}

\section{Monotonicity of the solution}\label{sec:mono}

In this section, we establish the monotonicity of the solution, which
is the main technical step in the proof of our main theorem.

\begin{theorem}\label{thm:monotonicity}
Let $n\ge2$, and let $u\in C^2(\R^n)$ be a bounded, non-constant solution of
\[
-\Delta u=u-u^3\qquad\text{in }\R^n.
\]
If $u<0$ in $\{x_n<0\}$, then $\partial_{x_n}u>0$ in $\R^n$.
\end{theorem}

\begin{proof}

By \eqref{eq:comparison}, after a vertical translation, we may assume that
\begin{equation}\label{eq:normalization}
-1<u(x',x_n)\le H(x_n)\qquad\text{in }\R^n.
\end{equation}
Interior elliptic estimates, or alternatively Modica  inequality \cite{Modica}, give
a constant $L$ such that $|\nabla u|\le L$.

For $\lambda\in\R$ and a function $f$ defined on $\R^n$, set
\[
D_\lambda[f](x',x_n)
=f(x',x_n)-f(x',2\lambda-x_n),
\qquad x_n<\lambda,
\]
and
\[
M_\lambda
=\sup_{\substack{x'\in\R^{n-1}\\x_n<\lambda}}
D_\lambda[u](x',x_n).
\]
Let $\lambda_0=H^{-1}(-1/\sqrt3)$.

\medskip

\noindent\textbf{Claim 1.}
We have $M_\lambda\le0$ for every $\lambda\le\lambda_0$.

\begin{proof}
Suppose that $M_\lambda>0$. Choose a sequence
$y_j=(y'_j,y_{j,n})\in\R^n$, with $y_{j,n}<\lambda$, such that
\[
D_\lambda[u](y'_j,y_{j,n})\longrightarrow M_\lambda.
\]
The bounds
\[
D_\lambda[u](x',x_n)
\le u(x',x_n)+1
\le H(x_n)+1
\]
and
\[
|D_\lambda[u](x',x_n)|
\le2L(\lambda-x_n)
\]
show that, for sufficiently large $j$, the heights $y_{j,n}$ are
bounded below and stay a positive distance from $\lambda$.

By interior elliptic estimates, after passing to a subsequence, we have
\[
y_{j,n}\rightarrow s<\lambda,
\qquad
u(x'+y'_j,x_n)\rightarrow V(x',x_n)
\quad\text{in }C^2_{\loc}(\R^n).
\]
The limit $V$ solves the Allen--Cahn equation and satisfies
\begin{equation}\label{eq:limit-bounds}
-1\le V(x',x_n)\le H(x_n),
\qquad |\nabla V|\le L.
\end{equation}
Moreover,
\begin{equation}\label{eq:attained}
D_\lambda[V](0,s)=M_\lambda,
\qquad
D_\lambda[V]\le M_\lambda
\quad\text{in }\{x_n<\lambda\}.
\end{equation}

Subtracting the equations satisfied by $V$ and its reflection gives
\[
-\Delta D_\lambda[V]
=
\left(
1-V(x',x_n)^2
-V(x',x_n)V(x',2\lambda-x_n)
-V(x',2\lambda-x_n)^2
\right)D_\lambda[V].
\]
Since $D_\lambda[V]$ attains its positive maximum at $(0,s)$, we have
$\Delta D_\lambda[V](0,s)\le0$, and hence
\begin{equation*}
V(0,s)^2
+V(0,s)V(0,2\lambda-s)
+V(0,2\lambda-s)^2
\le1.
\end{equation*}
This implies that $
V(0,s)>-\frac{1}{\sqrt3}$.
By \eqref{eq:limit-bounds} and the strict monotonicity of $H$, it follows that
\[
\lambda_0<s<\lambda.
\]
Thus $M_\lambda>0$ is impossible when $\lambda\le\lambda_0$.
Consequently,
\[
D_\lambda[u]\le0
\qquad\text{in }\{x_n<\lambda\}
\quad\text{for every }\lambda\le\lambda_0.
\]
\end{proof}

Now define $
\lambda_*=
\sup\left\{\lambda |
M_\mu \le0 
\text{ for every }\mu\le\lambda\right\}.$
We have shown above that $\lambda_*\ge \lambda_0$.
\medskip

\noindent \textbf{Claim 2:} We have $\lambda_* = \infty$.
\begin{proof}
Suppose, for contradiction, that $\lambda_*<\infty$.
By continuity,
\begin{equation}\label{eq:critical-order}
D_{\lambda_*}[u]\le0.
\end{equation}
By the definition of $\lambda_*$, there exists a sequence $\lambda_k>\lambda_*$
with $\lambda_k\to\lambda_*$ such that
\[
M_k:=M_{\lambda_k}>0.
\]

For $x_n<\lambda_*$, we have
\[
\begin{aligned}
D_{\lambda_k}[u](x',x_n)
&=D_{\lambda_*}[u](x',x_n)
 +u(x',2\lambda_*-x_n)-u(x',2\lambda_k-x_n)\\
&\le2L(\lambda_k-\lambda_*).
\end{aligned}
\]
For $\lambda_*\le x_n<\lambda_k$, the same upper bound follows from Modica's inequality. Therefore,
\begin{equation}\label{eq:small-maxima}
0<M_k\le2L(\lambda_k-\lambda_*)\longrightarrow0.
\end{equation}

Applying the same argument as in the proof of Claim 1, we obtain $U_k$ and $x_{k,n}$ such that
\begin{gather}
D_{\lambda_k}[U_k](0,x_{k,n})=M_k,
\qquad \lambda_0<x_{k,n}<\lambda_k,\label{eq:chosen-maxima}\\
U_k(0,x_{k,n})>-\frac1{\sqrt3}.\label{eq:chosen-values}
\end{gather}
By compactness, after passing to a subsequence, we obtain a limit $U$ such that
\[
U_k\to U\quad\text{in }C^2_{\loc}(\R^n),
\qquad
x_{k,n}\to x_{*,n}\in[\lambda_0,\lambda_*].
\]
The limit satisfies
\begin{equation}\label{eq:limit-properties}
U(x)\le H(x_n),\qquad U(0,x_{*,n})\ge-\frac1{\sqrt3}.
\end{equation}
In particular, $U$ is non-constant. It also follows from \eqref{eq:small-maxima} that
\begin{equation}\label{eq:limit-order}
D_{\lambda_*}[U]\le0\qquad\text{in }\{x_n<\lambda_*\}.
\end{equation}

We now prove that $D_{\lambda_*}[U]\equiv 0$.
If $x_{*,n}<\lambda_*$, equations \eqref{eq:small-maxima}
and \eqref{eq:chosen-maxima} imply
\[
D_{\lambda_*}[U](0,x_{*,n})=0.
\]
The strong comparison principle, applied to $U$ and its reflection,
gives
\begin{equation}\label{eq:symmetry}
U(x',x_n)=U(x',2\lambda_*-x_n).
\end{equation}

If $x_{*,n}=\lambda_*$, the interior maximum point in
\eqref{eq:chosen-maxima} satisfies
\[
0=\partial_{x_n}D_{\lambda_k}[U_k](0,x_{k,n})
=(\partial_{x_n} U_k)(0,x_{k,n})+(\partial_{x_n} U_k)(0,2\lambda_k-x_{k,n}).
\]
Passing to the limit gives
\begin{equation}\label{eq:zero-derivative}
\partial_{x_n}U(0,\lambda_*)=0.
\end{equation}
Define
\[
w(x',x_n)=U(x',2\lambda_*-x_n)-U(x',x_n),\qquad x_n<\lambda_*.
\]
By \eqref{eq:limit-order}, $w\ge0$. The strong minimum principle
and the Hopf lemma imply that, if $w\not\equiv0$ in the half-space, then
\[
\partial_{x_n}w(0,\lambda_*)<0.
\]
However, \eqref{eq:zero-derivative} implies
\[
\partial_{x_n}w(0,\lambda_*)=-2\partial_{x_n}U(0,\lambda_*)=0.
\]
This contradiction proves \eqref{eq:symmetry} in the boundary case as well.

Finally, \eqref{eq:limit-properties} and \eqref{eq:symmetry} give
\[
U(x',x_n)\le\min\{H(x_n),H(2\lambda_*-x_n)\}\le H(\lambda_*).
\]
Since $\lambda_*$ is finite,
\begin{equation}\label{eq:sup-less-one}
\sup_{\R^n}U\le H(\lambda_*)<1.
\end{equation}
On the other hand, $U$ is non-constant and satisfies
\[
U(x',x_n)\le H(x_n)<0\qquad(x_n<0).
\]
Applying Theorem~\ref{thm:comparison} to $U$ gives
$\sup_{\R^n}U=1$, contradicting
\eqref{eq:sup-less-one}. Hence $\lambda_*=\infty$.

\end{proof}

The strict inequality $\partial_{x_n} u >0$ now follows from the
strong maximum principle applied to the equation satisfied by
$\partial_{x_n} u$. Indeed, if $u_{x_n}\equiv0$, then
\eqref{eq:normalization} would imply
$
u(x)\le\inf_{x_n\in\R}H(x_n)=-1,
$
contradicting $u>-1$.
\end{proof}

\section{Proof of the main theorem}\label{sec:main}

The following lemma shows that monotonicity in the $x_n$ direction,
together with one-sided uniform convergence, implies finite energy
density at infinity. The result follows from \cite[Theorem~1.3]{AC},
but we include the proof for completeness.

\begin{lemma}\label{lem:energy}
Suppose that $u$ is a bounded entire Allen--Cahn solution with $u_{x_n}>0$
and
\[
u(x',x_n)\rightarrow-1\quad\text{uniformly in }x'
\quad\text{as }x_n\to-\infty.
\]
Then, there exists a constant $C>0$ independent of $R\ge1$ such that
\[
\mathcal E(u;B_R)\le CR^{n-1}.
\]
\end{lemma}

\begin{proof}
Let $L:=\|\nabla u\|_{L^\infty(\R^n)}<\infty$.
Using the equation, integration by parts, and $u_{x_n}>0$, we obtain
\[
\frac{d}{dt}\mathcal E(u;B_R+te_n)
=\int_{\partial(B_R+te_n)}u_{x_n}\partial_\nu u\,dS
\le L\frac{d}{dt}\int_{\partial B_R}u(x',x_n+t)\,dS.
\]
The uniform convergence as $x_n \rightarrow -\infty$ implies that $u|_{B_R + te_n}$ converges to $-1$ as $t \rightarrow -\infty$ in $C^1$-topology. This implies that
$\mathcal E(u;B_R+te_n)\rightarrow 0$ as $t\rightarrow-\infty$.
Integrating the inequality from $-\infty$ to $0$ therefore gives
\[
\mathcal E(u;B_R)
\le L\int_{\partial B_R}(u+1)\,dS
\le CR^{n-1}.
\]
\end{proof}

\begin{proof}[Proof of Theorem~\ref{thm:all}]
For $n=2$, the conclusion follows from \cite{HLSWW}.
We therefore assume $n\ge3$.

The lower half-space is connected and contains no zero of $u$.
Changing the sign of $u$, if necessary, we may assume that $u<0$ there.
The comparison theorem, followed by a vertical translation, gives
\[
-1<u(x',x_n)\le H(x_n).
\]
Theorem~\ref{thm:monotonicity} yields $u_{x_n}>0$.
Moreover, the comparison implies that $u(x',x_n)\to-1$ uniformly
in $x'$ as $x_n\to-\infty$. Lemma~\ref{lem:energy} therefore gives
$
\mathcal E(u;B_R)\le CR^{n-1}$ for $R\ge1$. Set
$
\Theta(R)=R^{1-n}\mathcal E(u;B_R).
$ As $\Theta$ is non-decreasing  the
energy bound implies
\[
0<\Theta(1)\le\Theta_\infty
:=\lim_{R\to\infty}\Theta(R)<\infty.
\]

Choose $R_j\uparrow\infty$ and set
\[
\varepsilon_j=R_j^{-1},\qquad v_j(x)=u(R_jx).
\]
Then $\varepsilon_j\Delta v_j=\varepsilon_j^{-1}F'(v_j)$,
$-1<v_j<1$, and, for every fixed $r>0$,
\[
\int_{B_r}
\left(
\frac{\varepsilon_j}{2}|\nabla v_j|^2
+\frac{F(v_j)}{\varepsilon_j}
\right)\,dx
=R_j^{1-n}\mathcal{E}(u,B_{R_jr})
=r^{n-1}\Theta(R_jr).
\]
Applying Theorem~\ref{thm:HT} on balls and taking a diagonal
subsequence gives a stationary integral varifold $V$ on $\R^n$
satisfying \eqref{eq:energy-measure} and \eqref{eq:equipartition} locally.
Hence we have
\[
\sigma\|V\|(B_r)=\Theta_\infty r^{n-1} \quad \text{ and } \quad \|V\|(\partial B_r)=0 \qquad \text{for every $r>0$}.
\]
and that $V$ is a cone centered at the origin.
Since $\Theta_\infty>0$, this cone is nonzero.
See also \cite[Proposition~11.2]{Wang}.

Let $K\Subset\{x_n<0\}$ and choose $\eta>0$ such that
$x_n\le-\eta$ on $K$.
The comparison theorme \ref{thm:comparison}  gives
\[
0<1+v_j(x)\le Ce^{-\sqrt2\eta R_j}\qquad(x\in K).
\]
Now the Modica inequality yields
\[
0\le
\frac{\varepsilon_j}{2}|\nabla v_j(x)|^2
+\frac{F(v_j(x))}{\varepsilon_j}
\le \frac{2F(v_j(x))}{\varepsilon_j}
\le CR_j e^{-2\sqrt{2}\eta R_j}
\longrightarrow 0
\]
uniformly on $K$. Consequently, $\operatorname{spt}\|V\|\subset\{x_n\ge0\}$.
Since $V$ is a stationary integral cone,
\cite[Theorem~36.5 and Remark~36.6]{Simon} and the constancy theorem \cite[Theorem~41.1]{Simon} implies that
\[
V=m|P|,\qquad m\in\mathbb N,\quad m\ge1,
\]

\medskip
\noindent\textbf{Claim:} We have $m=1$.
\begin{proof}
Define
\[
\Phi(s)=\int_{-1}^s\sqrt{2F(a)}\,da.
\]
Since $(v_j)_{x_n}>0$, the Modica inequality gives
\[
0\le\varepsilon_j(v_j)_{x_n}\le\sqrt{2F(v_j)},
\qquad
\frac{\varepsilon_j}{2}|\nabla v_j|^2
\le\frac{F(v_j)}{\varepsilon_j}.
\]
It follows that
\begin{equation}\label{eq:flux}
0\le
\frac{\varepsilon_j}{2}|\nabla v_j|^2
+\frac{F(v_j)}{\varepsilon_j}
-\partial_{x_n}\Phi(v_j)
\le
\frac{F(v_j)}{\varepsilon_j}
-\frac{\varepsilon_j}{2}|\nabla v_j|^2
+\varepsilon_j|\nabla_{x'}v_j|^2.
\end{equation}
Indeed, the lower bound follows from Young's inequality, while
the upper bound follows from
\[
\begin{aligned}
&\frac{F(v_j)}{\varepsilon_j}
-\frac{\varepsilon_j}{2}|\nabla v_j|^2
+\varepsilon_j|\nabla_{x'}v_j|^2
-\left(
\frac{\varepsilon_j}{2}|\nabla v_j|^2
+\frac{F(v_j)}{\varepsilon_j}
-\partial_{x_n}\Phi(v_j)
\right)\\
&\qquad
=(v_j)_{x_n}
\bigl(\sqrt{2F(v_j)}-\varepsilon_j(v_j)_{x_n}\bigr)
\ge0.
\end{aligned}
\]
We next show that the tangential gradient energy converges to 0 in $L^1_{\loc}(\R^n)$-topology.
Recall that an element $S\in G(n,n-1)$ is identified with the
orthogonal projection onto the corresponding $(n-1)$-dimensional
plane. On $\{|\nabla v_j|>0\}$, set
\[
S_j(x)=I-\nu_j(x)\otimes\nu_j(x)
\]
where $\nu_j = \frac{\nabla v_j}{|\nabla v_j|}$. 

Hence, for every $\zeta\in C_c(\R^n)$, we have
\[
\begin{aligned}
\sigma\int_{\R^n\times G(n,n-1)}
\zeta(x)|Se_n|^2\,dV_j(x,S)
&=
\int_{\{|\nabla v_j|>0\}}
\zeta(x)|S_j(x)e_n|^2
\varepsilon_j|\nabla v_j|^2\,dx\\
&=
\int_{\R^n}
\zeta(x)\varepsilon_j|\nabla_{x'}v_j|^2\,dx.
\end{aligned}
\]

Let $\pi_P$ denote the orthogonal projection onto
$P=\{x_n=0\}$. The varifold convergence
$V_j\to m|P|$ yields
\[
\begin{aligned}
\int_{\R^n}
\zeta\,\varepsilon_j|\nabla_{x'}v_j|^2\,dx
&\longrightarrow
\sigma m\int_P\zeta(x)|\pi_Pe_n|^2\,d\mathcal H^{n-1}(x) = 0.
\end{aligned}
\]
Choosing $\zeta\ge1$ on any prescribed compact set gives
\[
\varepsilon_j|\nabla_{x'}v_j|^2
\longrightarrow0
\quad\text{in }L^1_{\loc}(\R^n).
\]

Together with \eqref{eq:equipartition} and the nonnegative
upper and lower bounds in \eqref{eq:flux}, this implies
\[
\frac{\varepsilon_j}{2}|\nabla v_j|^2
+\frac{F(v_j)}{\varepsilon_j}
-\partial_{x_n}\Phi(v_j)
\longrightarrow0
\quad\text{in }L^1_{\loc}(\R^n).
\]

Let $0\le\varphi\in C_c^1(\R^{n-1})$ satisfy
$\int_{\R^{n-1}}\varphi\,dx'>0$, and let
$\psi\in C_c^1(\R)$ satisfy $0\le\psi\le1$ and $\psi(0)=1$.
For each fixed $x'$, we have
\[
\partial_{x_n}\Phi(v_j)
=\sqrt{2F(v_j)}(v_j)_{x_n}\ge0.
\]
Moreover, $|v_j|\le1$ implies
$0\le\Phi(v_j)\le\Phi(1)=\sigma$.
Thus,
\[
\int_\R\partial_{x_n}\Phi(v_j(x',x_n))\,dx_n
=
\lim_{R\to\infty}
\left[
\Phi(v_j(x',R))-\Phi(v_j(x',-R))
\right]
\le\sigma.
\]
Since $0\le\psi\le1$, it follows that
\[
\int_\R
\psi(x_n)\partial_{x_n}\Phi(v_j(x',x_n))\,dx_n
\le\sigma.
\]
Multiplying by $\varphi(x')$ and integrating in $x'$, we obtain
\[
\int_{\R^n}
\varphi(x')\psi(x_n)\partial_{x_n}\Phi(v_j)\,dx
\le\sigma\int_{\R^{n-1}}\varphi\,dx'.
\]

On the other hand, \eqref{eq:energy-measure} and $V=m|P|$ give
\[
\begin{aligned}
&\lim_{j\to\infty}
\int_{\R^n}\varphi(x')\psi(x_n)
\left(
\frac{\varepsilon_j}{2}|\nabla v_j|^2
+\frac{F(v_j)}{\varepsilon_j}
\right)\,dx\\
&\qquad
=\sigma m\int_P\varphi(x')\psi(0)\,d\mathcal H^{n-1}
=\sigma m\int_{\R^{n-1}}\varphi\,dx'.
\end{aligned}
\]
The preceding $L^1_{\loc}$ convergence allows us to replace
the energy density in this limit by $\partial_{x_n}\Phi(v_j)$.
Consequently,
\[
\sigma m\int_{\R^{n-1}}\varphi\,dx'
=
\lim_{j\to\infty}
\int_{\R^n}
\varphi(x')\psi(x_n)\partial_{x_n}\Phi(v_j)\,dx
\le\sigma\int_{\R^{n-1}}\varphi\,dx'.
\]
Since $\int\varphi>0$, we conclude that $m\le1$.
This proves that $m=1$.
\end{proof}

Now we have $\Theta_\infty=\sigma\omega_{n-1}$
where $\omega_{n-1}$ is the volume of the unit ball in $\R^{n-1}$.
Wang's rigidity theorem
\cite[Theorem~11.1]{Wang} therefore yields
\[
u(x)=H(e\cdot x+c)
\]
for a unit vector $e=(e',e_n)$ and a constant $c$.
Now it is straightforward to check that
the original solution $u$ is either of $\pm H(x_n-a)$ for some $a>0$.
\end{proof}

\bibliographystyle{abbrv}
\bibliography{half_space_references}

\end{document}